\pdfoutput=1  
\documentclass[10pt]{amsart}
\title{Weak Proregularity,  Derived Completion, Adic Flatness, and Prisms}
\date{13 April 2021} 

\usepackage{amssymb}
\usepackage[all]{xy}
\usepackage{graphicx}
\usepackage{comment}

\usepackage{newtxtext}
\usepackage[scaled=0.94]{roboto}  
\usepackage[varbb, varvw]{newtxmath}
\usepackage[scaled=0.95]{inconsolata} 
\usepackage[cal=cm, calscaled=0.95, frak=euler, frakscaled=0.98]{mathalfa}
\usepackage[T1]{fontenc}

\usepackage{hyperref}
\hypersetup{colorlinks=false}

\author{Amnon Yekutieli}
\address{Department of  Mathematics,
Ben Gurion University, Be'er Sheva 84105, Israel}
\email{amyekut@math.bgu.ac.il}

\newtheorem{thm}[equation]{Theorem}
\newtheorem{cor}[equation]{Corollary}
\newtheorem{prop}[equation]{Proposition}
\newtheorem{lem}[equation]{Lemma}
\theoremstyle{definition}
\newtheorem{dfn}[equation]{Definition}
\newtheorem{rem}[equation]{Remark}
\newtheorem{exa}[equation]{Example}

\numberwithin{equation}{section}
\newcommand{\iso}{\xrightarrow{%
\smash{\raisebox{-0.5ex}{\ensuremath{\scriptstyle \simeq  \mspace{2mu}}}}}}

\newcommand{\xar}{\xrightarrow}
\newcommand{\lto}{\leftarrow}
\newcommand{\opn}{\operatorname}
\newcommand{\cat}[1]{\operatorname{\mathsf{#1}}}

\newcommand{\cd}{\mspace{1.3mu}{\cdot}\mspace{1.3mu}}

\newcommand{\rmitem}[1]{\item[\text{\textup{(#1)}}]}
\newcommand{\mfrak}[1]{\mathfrak{#1}}

\newcommand{\mrm}[1]{\mathrm{#1}}

\newcommand{\Ga}{\Gamma}
\newcommand{\La}{\Lambda}
\newcommand{\si}{\sigma}

\newcommand{\de}{\delta}

\newcommand{\ep}{\epsilon}

\newcommand{\p}{\mfrak{p}}

\newcommand{\m}{\mfrak{m}}

\renewcommand{\a}{\mfrak{a}}
\renewcommand{\b}{\mfrak{b}}
\renewcommand{\c}{\mfrak{c}}

\renewcommand{\aa}{\bsym{a}}
\newcommand{\bb}{\bsym{b}}

\renewcommand{\tt}{\bsym{t}}
\renewcommand{\ss}{\bsym{s}}

\newcommand{\K}{\mathbb{K}}

\newcommand{\Q}{\mathbb{Q}}
\newcommand{\Z}{\mathbb{Z}}
\newcommand{\N}{\mathbb{N}}

\newcommand{\tup}[1]{\textup{#1}}
\newcommand{\bsym}[1]{\boldsymbol{#1}}

\newcommand{\ot}{\otimes}

\newcommand{\what}[1]{\widehat{#1}}

\renewcommand{\d}{\mathrm{d}}

\newcommand{\sub}{\subseteq}

\newcommand{\lsp}{\mspace{1.5mu}}
\newcommand{\mbar}[1]{\skew{3.5}\bar{#1}} 

\begin{document}

\begin{abstract}
We begin by recalling the role that {\em weak proregularity} of an 
ideal in a commutative ring has in {\em derived completion} and {\em adic 
flatness}. We also introduce the new concepts of {\em idealistic} and {\em 
sequential} {\em derived completion}, and prove a few results about them, 
including the fact that these two concepts agree iff the ideal is weakly 
proregular. Next we study the {\em local nature of weak proregularity}, and its 
behavior w.r.t.\ {\em ring quotients}. These results allow us to prove our main 
theorem, which states that {\em weak proregularity occurs in the context of 
bounded prisms}. Prisms belong to the new groundbreaking theory of 
{\em perfectoid rings}, developed by Scholze and his collaborators. Since  
perfectoid ring theory makes heavy use of derived completion and adic flatness, 
we anticipate that our results shall help simplify and improve some of the more 
technical aspects of this theory.
\end{abstract}

\maketitle

\tableofcontents

\setcounter{section}{-1}
\section{Introduction}

Let $A$ be a commutative ring. Various authors have previously considered two 
kinds of derived completions of complexes of $A$-modules, without a careful 
distinction between them, leading to some confusion. We begin our paper by 
recalling these two kinds of derived completions, and giving them names. 

The first kind of derived completion is determined by a finitely 
generated ideal $\a \sub A$, and accordingly we call it {\em idealistic derived 
completion}. The idealistic derived completion of a complex of $A$-modules $M$ 
is the complex $\mrm{L} \La_{\a}(M)$, where $\mrm{L} \La_{\a}$ is the left 
derived functor of the $\a$-adic completion functor $\La_{\a}$. 
Idealistic derived completion appeared, perhaps for the first time, in the 
paper \cite{AJL}. It was also studied in \cite{PSY}, among other 
papers. The functor $\mrm{L} \La_{\a}$ is straightforward in its 
definition, yet its properties are not so easy to understand, and often it is  
not as nice as could be expected. See Section 1. 

In section 2 we define {\em sequential derived completion}, associated 
to a finite sequence $\aa = (a_1, \ldots, a_n)$ of elements in $A$. 
The sequence $\aa$ gives rise to a complex of $A$-modules 
$\opn{K}^{\vee}_{\infty}(A; \aa)$,
called the {\em infinite dual Koszul complex}, or the {\em augmented \v{C}ech 
complex}. For a complex of $A$-modules $M$, its 
sequential derived $\aa$-adic completion is the complex 
$\opn{RHom}_{A} \bigl( \opn{K}^{\vee}_{\infty}(A; \aa), M \bigr)$. 
This kind of derived completion was studied in the papers 
\cite{AJL}, \cite{DG}, \cite{KS2} \cite{PSY}, \cite{BS1}, \cite{BS2} and many 
more, as well in Chapter 4 of the book \cite{Lu}, and in  
\href{https://stacks.math.columbia.edu/tag/091N}{Section tag=\texttt{091N}}
of the online resource \cite{SP}. The sequential derived completion is quite 
puzzling; it is not immediately clear how this operation is related to derived 
$\a$-adic completion, where $\a \sub A$ is the ideal generated by $\aa$. Indeed 
the relation is indirect: according to Proposition \ref{prop:316}
(a result that is implicit in \cite{PSY}) there is a 
canonical morphism of triangulated functors 
\begin{equation} \label{eqn:250}
\opn{RHom}_{A} \bigl( \opn{K}^{\vee}_{\infty}(A; \aa), - \bigr)
\to \mrm{L} \La_{\a}  
\end{equation}
from $\cat{D}(A)$ to itself. 
Despite the not-so-obvious definition, the sequential derived completion
functor is quite easy to analyze, and it always has nice properties. 

Besides the idealistic and sequential derived completion functors, there 
are also the {\em idealistic} and {\em sequential derived torsion functors}. 
These are also discussed in Sections 1-2. 

Section 3 contains the definition of {\em weak proregularity} of a
finite sequence $\aa$ of elements in $A$ (it is Definition \ref{dfn:45}). An 
ideal $\a \sub A$ is called weakly proregular (WPR) if it is generated by some 
WPR sequence; and in this case, every finite sequence that generates $\a$ is 
WPR. Weak proregularity of the ideal $\a$ is a subtle weakening of the 
noetherian condition on the ring $A$ (in the sense that when $A$ is noetherian, 
every ideal $\a$ in it is WPR). Though first studied by Grothendieck \cite{LC} 
in the 1960's, the name weak proregularity was coined by Lipman 
\cite[Correction]{AJL} around 1999, and most of the work on it is more recent. 

The next theorem gives a first indication of the importance of weak 
proregularity. Note that our {\em rings are not assumed to be noetherian};
indeed, most results here are uninteresting for noetherian rings. 

\begin{thm} \label{thm:400}
Let $A$ be a commutative ring, let $\aa$ be a finite sequence in $A$, and let 
$\a \sub A$ be the ideal generated by $\aa$. The following two conditions are 
equivalent. 
\begin{itemize}
\rmitem{i} The sequence $\aa$ is weakly proregular. 

\rmitem{ii} The morphism of functors \tup{(\ref{eqn:250})} is an isomorphism. 
\end{itemize}
\end{thm}

In plain words, the ideal $\a$ is WPR iff the idealistic derived $\a$-adic 
completion coincides with the sequential derived $\aa$-adic completion.
This is part of Theorem \ref{thm:165} in the body of the paper. 
The implication (i) $\Rightarrow$ (ii) was proved in \cite{PSY}, and the 
implication (ii) $\Rightarrow$ (i) was very recently proved by Positselski 
\cite{Po}.

In Section 4 we look at the concept of {\em $\a$-adic flatness}. This is a 
variant of the usual notion of flatness, and some texts (e.g.\ \cite{BS2}) call 
it {\em $\a$-complete flatness}. 

Here are two useful theorems about $\a$-adic flatness.

\begin{thm} \label{thm:401}
Let $A$ be a commutative ring, let $\a$ be a weakly proregular ideal in $A$, 
and let $M$ be an $\a$-adically flat $A$-module, with $\a$-adic completion 
$\what{M} = \La_{\a}(M)$. Then the $A$-module $\what{M}$ is $\a$-adically flat. 
\end{thm}

This is Theorem \ref{thm:227} in the body of the paper, and it was proved in 
\cite{Ye2}.

Recall that when the ring $A$ is noetherian, all ideals in it are WPR. 
In this case the two notions of flatness coincide:

\begin{thm} \label{thm:402}
If $A$ is a commutative noetherian ring, $\a$ is an ideal in $A$, and 
$\what{M}$ is an $\a$-adically flat $\a$-adically complete $A$-module,
then $\what{M}$ is a flat $A$-module. 
\end{thm}

This is Theorem \ref{thm:228} in the body of the paper, and it was also proved 
in \cite{Ye2}. This result was used in a recent paper of Cesnavicius and Scholze 
-- see \cite[Lemma 6.2.1]{CS}. 

The following theorem studies the behavior of WPR when passing to quotient 
rings. An element $a \in A$ is called WPR if it is so as a length $1$ sequence. 

\begin{thm} \label{thm:404}
Let $A$ be a commutative ring, and let $a, b \in A$. Assume that $a$ is 
a regular element of $A$, and the image $\bar{b}$ of $b$ in the ring 
$\mbar{A} := A / (a)$ is a weakly proregular element. Then the length $2$ 
sequence $\aa := (a, b)$ in $A$ is weakly proregular.
\end{thm}

This is Theorem \ref{thm:105} in the body of the paper, and its proof is 
quite long. See Remark \ref{rem:320} regarding possible generalizations of this 
theorem. 

The next theorem (it is Theorem \ref{thm:306} in the body of the paper) shows 
that weak proregularity of an ideal $\a \sub A$ is local on $\opn{Spec}(A)$. By 
a covering sequence $\ss = (s_1, \ldots, s_n)$ of $A$ we mean a sequence of 
elements such 
that $\opn{Spec}(A) = \bigcup\nolimits_{i} \opn{Spec}(A_{s_i})$.

\begin{thm} \label{thm:405}
Let $A$ be a commutative ring and let $\a \sub A$ be an ideal. The following 
two conditions are equivalent\tup{:}
\begin{itemize}
\rmitem{i} The ideal $\a$ is weakly proregular.  

\rmitem{ii} There is a covering sequence $\ss = (s_1, \ldots, s_n)$  of $A$, 
such that for every $i$ the ideal $\a_{s_i} \sub A_{s_i}$
is weakly proregular. 
\end{itemize}
\end{thm}

The concept of {\em $p$-adic prism} was defined in the recent paper \cite{BS2} 
by Bhatt and Scholze, and we reproduce it in Definition \ref{dfn:135}.
They also defined {\em bounded prisms} (see Definition \ref{dfn:136}) and 
{\em perfect prisms}.  

Prisms are closely related to {\em perfectoid rings}.
Indeed, if $(A, I)$ is a perfect prism, then the ring $A / I$ is 
perfectoid, and this operation is an equivalence from the category 
of perfect prisms to the category of perfectoid rings. 
There is also a {\em perfection functor} from prisms to perfect prisms. See 
\cite[{Example 1.3$\mrm{(2)}$}]{BS2}. 

The algebraic theory of {\em perfectoid rings}, and its geometric companion 
theory of {\em perfectoid spaces}, have revolutionized arithmetic geometry in 
recent years. Much of the technical part of perfectoid theory involves derived 
completion (in the sequential sense, as defined here) and adic flatness (also 
known as complete flatness). 
Scholze and his coauthors have discovered a variant of our Theorem \ref{thm:401}
for the case of a principal ideal $\a = (a)$, see e.g.\ \cite[Lemma 4.4]{BMS}; 
but the more general result is also needed for applications in this area. 
As already mentioned, our Theorem \ref{thm:402} was already used by Scholze et 
al. 

Here is our main theorem (it is Theorem \ref{thm:135} in the body of the 
paper). It allows the application of all existing results on weakly proregular 
ideals to bounded prisms. 

\begin{thm} \label{thm:144}
Suppose $(A, I)$ is a bounded $p$-adic prism. Then the ideal 
$\a := I + (p) \sub A$ is weakly proregular.
\end{thm}

To end the Introduction, let us mention three papers that have cited earlier 
versions of our paper. The paper \cite{Po} by Positselski was written, to some 
extent, in response to our present work. It looks at $\a$-adic completion 
(plain and derived) from another perspective, and we recommend reading it.  
Some of the results stated in our paper are either proved, or
explained with precise references, in \cite{Po} (cf.\ Theorem \ref{thm:400} 
above). 

The second reference citing an earlier version of this paper is \cite{Ch} by 
Chatzistamatiou. Proposition 3.1.5 of \cite{Ch} is a variation of our Theorem 
\ref{thm:401}. Roughly speaking, Chatzistamatiou replaces our assumption 
that the ideal $\a$ is WPR by the assumption that $\a = I + (p)$  for a 
$p$-adic prism $(A, I)$, this prism is not necessarily bounded, but the 
ideal $I$ is principal. For a bounded prism (this is the main application in 
\cite{Ch}), his result is a consequence of our Theorems \ref{thm:401} and 
\ref{thm:144}.

The recent paper \cite{Sn} by Schenzel has a large overlap with our  present 
paper. In particular, Schenzel appears to reprove our Theorems \ref{thm:404}, 
\ref{thm:405} and \ref{thm:144}, possibly under slightly modified conditions. 
Schenzel had been exposed to earlier versions of our paper several months ago, 
and we even had some email exchanges about this material.
But it seems Schenzel does not give proper attribution to our work 
in \cite{Sn}.

\medskip 
\noindent
{\em Acknowledgments}. 
The author wishes to thank Leonid Positselski for providing him with many 
helpful suggestions, and for sharing an early version of his paper \cite{Po}. 
Thanks also to Bhargav Bhatt, Johan de Jong, Pierre Schapira, Liran Shaul,
Peter Scholze and Andre Chatzistamatiou for useful comments, to Asaf Yekutieli 
for mentioning the paper \cite{BS2}, and to Daniel Disegni for asking a 
good question at just the right time. Lastly the author thanks one of the 
anonymous referees of an earlier version of the paper, and the anonymous referee 
of the current version, for reading our paper with care and proposing some 
improvements. This work was supported by the Israel Science Foundation grant 
no.\ 824/18.

\section{Idealistic Derived Completion and Torsion} 
\label{sec:idealistic}

Throughout the paper $A$ is a commutative ring.
It is {\em not} assumed that $A$ is a noetherian ring.
We denote by $\cat{M}(A)$ the category of $A$-modules, and by 
$\cat{D}(A)$ its unbounded derived category. 
Our reference for derived categories is the book \cite{Ye3}. 

In this section, $\a$ is a finitely generated ideal of the ring $A$. 

For an $A$-module $M$, its $\a$-adic completion is
\[ \La_{\a}(M) := 
\lim_{\leftarrow k} \, (M \lsp / \lsp \a^{k} \cd M) . \]
Completion is an $A$-linear functor 
\begin{equation} \label{eqn:180}
\La_{\a} : \cat{M}(A) \to \cat{M}(A) , 
\end{equation}
which is neither left nor right exact. 
It is an idempotent functor, in the sense that the two canonical morphisms 
$\La_{\a} \to \La_{\a} \circ \La_{\a}$
are isomorphism (cf.\ \cite[Definition 2.8]{VY}, and Remark \ref{rem:180}).
There is a functorial homomorphism 
$\tau_{\a, M} : M \to \La_{\a}(M)$,
and we call $M$ an {\em $\a$-adically complete module} if $\tau_{\a, M}$ is an 
isomorphism. (Note that older texts used the adjective ``complete and 
separated'' for what we call {complete.)

For an $A$-module $M$ we also have its $\a$-torsion submodule
\[ \Ga_{\a}(M) := \lim_{k \to} \, \opn{Hom}_A(A / \a^{k}, M) . \]
Torsion is a left exact $A$-linear functor 
\begin{equation} \label{eqn:181}
\Ga_{\a} : \cat{M}(A) \to \cat{M}(A) ,
\end{equation}
and it too is idempotent. 
There is a functorial homomorphism 
$\si_{\a, M} : \Ga_{\a}(M) \to M$,
and we call $M$ an {\em $\a$-torsion module} if $\si_{\a, M}$ is an 
isomorphism. 

\begin{dfn} \label{dfn:180}
The {\em idealistic derived $\a$-adic completion functor}
is the functor 
\[ \mrm{L} \La_{\a} : \cat{D}(A) \to \cat{D}(A) , \]
the left derived functor of the functor $\La_{\a}$ from (\ref{eqn:180}).
\end{dfn}

See Remark \ref{rem:210} regarding terminology. 
The left derived functor $\mrm{L} \La_{\a}$ is calculated by K-flat 
resolutions. It's first appearance seems to have been in the paper \cite{AJL}. 
Further study of the functor $\mrm{L} \La_{\a}$ was done in the paper 
\cite{PSY}. Several earlier papers (including \cite{Ma} and \cite{GM}) had 
considered the partial derived functors 
$\mrm{L}_q \La_{\a} = \mrm{H}^{-q} \circ \mrm{L} \La_{\a}$ 
from $\cat{M}(A)$ to itself. 

There is a functorial morphism 
\begin{equation} \label{eqn:182}
\tau_{\a, M}^{\mrm{L}} : M \to \mrm{L} \La_{\a}(M)
\end{equation}
in $\cat{D}(A)$; see \cite[Proposition 3.7]{PSY}.

\begin{dfn} \label{dfn:181} 
A complex $M \in \cat{D}(A)$ is called {\em derived $\a$-adically  
complete in the idealistic sense} if the morphism  
$\tau_{\a, M}^{\mrm{L}}$ in (\ref{eqn:182}) is an isomorphism. 
\end{dfn}

Here are the analogous definitions for $\a$-torsion. 

\begin{dfn} \label{dfn:185}
The {\em idealistic derived $\a$-torsion functor}
is the functor 
\[ \mrm{R} \Ga_{\a} : \cat{D}(A) \to \cat{D}(A) , \]
the right derived functor of the functor $\Ga_{\a}$ from 
(\ref{eqn:181}).
\end{dfn}

The derived functor $\mrm{R} \Ga_{\a}$ is calculated using K-injective 
resolutions. 

The idealistic derived torsion functor $\mrm{R} \Ga_{\a}$ has a long history, 
especially when the ring $A$ is local and $\a = \m$ is its maximal ideal. Then 
the cohomology modules 
$\opn{H}^q(\mrm{R} \Ga_{\m}(M)) = \mrm{R}^q \Ga_{\m}(M)$ are called 
the {\em local cohomologies} of $M$. See \cite{RD} and \cite{LC}. 

There is a functorial morphism 
\begin{equation} \label{eqn:185} 
\si_{\a, M}^{\mrm{R}} : \mrm{R} \Ga_{\a}(M) \to M 
\end{equation}
in $\cat{D}(A)$; see \cite[Proposition 3.10]{PSY}.

\begin{dfn} \label{dfn:186} 
A complex $M \in \cat{D}(A)$ is called {\em derived $\a$-torsion  
in the idealistic sense} if the morphism  
$\si_{\a, M}^{\mrm{R}}$ in (\ref{eqn:185}) is an isomorphism. 
\end{dfn}

The next proposition says that the idealistic derived completion and 
torsion functors depend not on the ideal $\a$, but rather on the closed subset 
in $\opn{Spec}(A)$ that the ideal $\a$ defines. 

\begin{prop} \label{prop:210}
Let $\a$ and $\b$ be finitely generated ideals in $A$ such that 
$\sqrt{\a} = \sqrt{\b}$. Then there are canonical isomorphisms 
$\mrm{L} \La_{\a} \iso \mrm{L} \La_{\b}$ and 
$\mrm{R} \Ga_{\a} \iso \mrm{R} \Ga_{\b}$
of triangulated functors from $\cat{D}(A)$ to itself, such that the diagrams 
\[ \UseTips \xymatrix @C=8ex @R=6ex {
\opn{Id} 
\ar[d]_{\tau^{\mrm{L}}_{\a}}
\ar[dr]^{\tau^{\mrm{L}}_{\b}}
\\
\mrm{L} \La_{\a}
\ar[r]^{\simeq}
&
\mrm{L} \La_{\b}
} 
\qquad \qquad 
\UseTips \xymatrix @C=8ex @R=6ex {
\mrm{R} \Ga_{\a}
\ar[r]^{\simeq}
\ar[dr]_{\si^{\mrm{R}}_{\a}}
&
\mrm{R} \Ga_{\b}
\ar[d]^{\si^{\mrm{R}}_{\b}}
\\
&
\opn{Id} 
} \]
are commutative.  
\end{prop}

\begin{proof}
An easy calculation shows that the obvious morphisms of functors 
$\La_{\a} \to \La_{\a + \b}$ and 
$\La_{\b} \to \La_{\a + \b}$ 
are both isomorphisms. It follows that there are isomorphisms
$\mrm{L} \La_{\a} \iso \mrm{L} \La_{\a + \b}$ 
and
$\mrm{L} \La_{\b} \iso \mrm{L} \La_{\a + \b}$ 
between the left derived functors. Hence there is a canonical isomorphism 
$\mrm{L} \La_{\a} \iso \mrm{L} \La_{\b}$.
The construction of $\tau^{\mrm{L}}_{\a}$ in \cite[Proposition 3.7]{PSY} shows 
that the first diagram above is commutative. 

The proof for derived torsion is similar.
\end{proof}

\begin{rem} \label{rem:210}
In the paper \cite{PSY} the names ``cohomologically $\a$-adically complete 
complex'' and ``cohomologically $\a$-torsion complex'' were used for the 
complexes in Definitions \ref{dfn:181} and \ref{dfn:186} respectively.  
The word ``cohomologically'' was replaced here with ``derived'', because it 
seems to describe the mathematical situation better: the condition in both 
cases is whether a certain morphism in the derived category $\cat{D}(A)$ is an 
isomorphism -- and not about the the cohomology $\opn{H}(M)$.  

The adjective ``idealistic'' was introduced in order to create a semantic 
distinction between the the definitions in this section, and those in the 
next section, which will be called ``sequential''. 
\end{rem}

\begin{rem} \label{rem:180}
We mentioned in passing that when $\a \sub A$ is a finitely generated ideal,
and $M$ is an arbitrary $A$-module, the $\a$-adic completion 
$\what{M} = \La_{\a}(M)$ is an $\a$-adically complete $A$-module. (This is 
included in the assertion that the functor $\La_{\a}$ is idempotent.)
See \cite[Corollary 3.6]{Ye1} for a proof. 

If the ideal $\a$ is not finitely generated, then the completion functor 
$\La_{\a}$ and the torsion functor $\Ga_{\a}$ are ill-behaved. Most disturbing 
is that fact that the functor $\La_{\a}$ could fail to be idempotent -- namely 
there are examples where the completion $\La_{\a}(M)$ is not 
$\a$-adically complete; see \cite[Example 1.8]{Ye1}.
\end{rem}

\begin{rem} \label{rem:250}
If the ideal $\a$ is weakly proregular, then the functor $\mrm{L} \La_{\a}$ is 
idempotent (see \cite[Lemma 7.9 and Proposition 7.10]{PSY}). Conversely, 
Positselski \cite[Proposition 5.3]{Po} shows that 
when $\a$ is not weakly proregular, and $P$ is a free $A$-module of infinite 
rank, then  $\mrm{L} \La_{\a}(P) = \La_{\a}(P)$ is not derived 
$\a$-adically complete in the idealistic sense. This means that without weak 
proregularity the functor $\mrm{L} \La_{\a}$ is not idempotent. 
\end{rem}

\section{Sequential Derived Completion and Torsion}
\label{sec:sequential}

Here again $A$ is a commutative ring. A finite sequence 
$\aa = (a_1, \ldots, a_n)$ of elements of $A$ gives rise to several complexes 
of $A$-modules. 

We start with a single element $a \in A$. For an $A$-module $M$
we denote by 
$\opn{mult}_M(a) : M  \to M$ the endomorphism
$\opn{mult}_M(a)(m) := a \cd m$. 

Recall that the {\em Koszul complex} associated to the element $a$ is 
\begin{equation} \label{eqn:101}
\opn{K}(A; a) := \bigl( \cdots \to 0 \to A \xar{\d} A \to 0 \to  
\cdots \bigr) , 
\end{equation}
concentrated in degrees $-1$ and $0$. 
The differential is $\d := \opn{mult}_A(a)$.
For $j \geq i$ in $\N$ there is a homomorphism of complexes 
\begin{equation} \label{eqn:197}
\mu_{j, i} : \opn{K}(A; a^{j}) \to \opn{K}(A; a^{i}) , 
\end{equation}
which is the identity in degree $0$, and $\opn{mult}_A(a^{j - i})$ in degree 
$-1$. 

Now consider a sequence of elements $\bsym{a} = (a_1, \ldots, a_n)$.
The associated Koszul complex is 
\begin{equation} \label{eqn:190}
\opn{K}(A; \bsym{a}) := 
\opn{K}(A; a_1) \ot_A \cdots \ot_A \opn{K}(A; a_n) . 
\end{equation}
This is a complex of finite rank free $A$-modules, concentrated in degrees 
$-n, \ldots, 0$. For $i \in \N$ let 
$\bsym{a}^i := (a^i_1, \ldots, a^i_n)$.
The homomorphism (\ref{eqn:197}) induces a homomorphism of complexes
$\mu_{j, i} : \opn{K}(A; \bsym{a}^j) \to \opn{K}(A; \bsym{a}^i)$,
making the collection of Koszul complexes 
$\bigl\{  \opn{K}(A; \bsym{a}^i) \bigr\}_{i \in \N}$
into an inverse system.

As an aside, let us mention that the Koszul complex $\opn{K}(A; \aa)$ is a 
{\em commutative DG ring} (in the sense of \cite[Definition 3.3.4]{Ye3}), and 
there is a DG ring homomorphism $A \to  \opn{K}(A; \aa)$. 
The inverse system 
$\bigl\{  \opn{K}(A; \bsym{a}^i) \bigr\}_{i \in \N}$
is of DG $A$-rings. 

The second complex we associate to the sequence $\aa$ is the {\em infinite 
dual Koszul complex}. It is
\begin{equation} \label{eqn:151}
\opn{K}^{\vee}_{\infty}(A; \aa) := 
\lim_{i \to} \, \opn{Hom}_{A} \bigl( \opn{K}(A; \aa^i), A \bigr) .
\end{equation}
This is a complex of flat $A$-modules, concentrated in degrees 
$0, \ldots, n$. 

The DG ring homomorphisms $A \to \opn{K}(A; \aa^i)$ induce, upon dualizing and 
passage to the limit, a homomorphism of complexes 
\begin{equation} \label{eqn:205}
\opn{K}^{\vee}_{\infty}(A; \aa) \to A . 
\end{equation}
We sometimes refer to this homomorphism as the {\em augmentation} of
$\opn{K}^{\vee}_{\infty}(A; \aa)$.

The complex $\opn{K}^{\vee}_{\infty}(A; \aa)$ has an alternative description. 
For a single element $a \in A$, its infinite dual Koszul complex admits this 
canonical isomorphism 
\begin{equation} \label{eqn:191}
\opn{K}^{\vee}_{\infty}(A; a) \cong 
\bigl( \cdots \to 0 \to A \xar{\d} A_a \to 0 \to  
\cdots \bigr) , 
\end{equation}
with $A$ in degree $0$, the localized ring $A_a = A[a^{-1}]$ in degree $1$, 
and the differential $\d$ is the ring homomorphism. For a sequence 
$\bsym{a} = (a_1, \ldots, a_n)$ there is an isomorphism 
\begin{equation} \label{eqn:195}
\opn{K}^{\vee}_{\infty}(A; \aa) \cong 
\opn{K}^{\vee}_{\infty}(A; a_1) \ot_A \cdots \ot_A 
\opn{K}^{\vee}_{\infty}(A; a_n) . 
\end{equation}

The third complex associated to the sequence $\aa$ is the {\em \v{C}ech complex}
$\opn{C}(A; \aa)$. This is also a bounded complex of flat $A$-modules.
This complex is more familiar in its algebro-geometric formulation, as follows.
Let $X := \opn{Spec}(A)$, let $U_i := \opn{Spec}(A[a_i^{-1}])$ be the principal 
affine open set in $X$ defined by the element $a_i$, and let 
$U := \bigcup_i U_i \sub X$. Then $\opn{C}(A; \aa)$ is the 
\v{C}ech complex associated to the affine open covering 
$\{ U_i \}_{i = 1, \ldots, n}$ of $U$. 
So $\opn{C}(A; \aa)$ is a complex of flat $A$-modules, concentrated in 
degrees $0, \ldots, n-1$.   
Actually, $\opn{C}(A; \aa)$ has a canonical structure of noncommutative central 
DG $A$-ring, see \cite[Section 8]{PSY}, and as such it is called the {\em 
derived localization of $A$ w.r.t.\ $\aa$}. 
When $n = 1$ and $a_1 = a$ this is a familiar commutative ring:
$\opn{C}(A; a) = A[a^{-1}]$. 

There is a canonical short exact sequence of complexes of $A$-modules 
\begin{equation} \label{eqn:40}
0 \to \opn{C}(A; \aa)[-1] \to \opn{K}^{\vee}_{\infty}(A; \aa) \to A \to 0 , 
\end{equation}
in which the augmentation (\ref{eqn:205}) occurs.
For this reason, the complex $\opn{K}^{\vee}_{\infty}(A; \aa)$
can also be called the {\em augmented \v{C}ech complex}. 

There is a fourth complex of $A$-modules that's canonically associated to $\aa$.
It is the {\em telescope complex} $\opn{Tel}(A; \aa)$, which is a complex of 
countable rank free $A$-modules, concentrated in degrees $0, \ldots, n$. 
The formula for $\opn{Tel}(A; \aa)$ is elementary but a bit messy; see 
\cite[Definition 5.1]{PSY}.
There is a canonical quasi-isomorphism 
\begin{equation} \label{eqn:166}
w_{\aa} : \opn{Tel}(A; \aa) \to \opn{K}^{\vee}_{\infty}(A; \aa) . 
\end{equation}
See \cite[Lemma 5.7]{PSY}.

The complexes $\opn{K}(A; \aa)$, $\opn{K}^{\vee}_{\infty}(A; \aa)$,
$\opn{C}(A; \aa)$ and $\opn{Tel}(A; \aa)$ are all {\em defined over $\Z$}, in 
the following sense. Consider the polynomial ring $\Z[\tt]$ in the sequence of 
variables $\tt = (t_1, \ldots, t_n)$. As a special case of the constructions 
above, there are complexes of $\Z[\tt]$-modules
$\opn{K}(\Z[\tt]; \tt)$, $\opn{K}^{\vee}_{\infty}(\Z[\tt]; \tt)$,
$\opn{C}(\Z[\tt]; \tt)$ and $\opn{Tel}(\Z[\tt]; \tt)$. The sequence 
$\aa =  (a_1, \ldots, a_n)$ in $A$ determines a ring 
homomorphism $\Z[\tt] \to A$ sending $t_i \mapsto a_i$, and under this 
homomorphism we 
obtain isomorphisms 
\begin{equation} \label{eqn:164}
\begin{aligned}
A \ot_{\Z[\tt]} \opn{K}(\Z[\tt]; \tt) & \iso \opn{K}(A; \aa) 
\\
A \ot_{\Z[\tt]} \opn{K}^{\vee}_{\infty}(\Z[\tt]; \tt) & \iso 
\opn{K}^{\vee}_{\infty}(A; \aa) 
\\
A \ot_{\Z[\tt]} \opn{C}(\Z[\tt]; \tt) & \iso \opn{C}(A; \aa) 
\\
A \ot_{\Z[\tt]} \opn{Tel}(\Z[\tt]; \tt) & \iso \opn{Tel}(A; \aa)
\end{aligned}
\end{equation}
of complexes of $A$-modules.

The next definitions are based on material from several papers and the book 
[LC]. A priori they appear to have no relation to their parallels in the 
previous section, beyond the hint we have in formula (\ref{eqn:215}). 
Understanding that they are indeed what their names suggest (e.g.\ derived 
torsion in Definition \ref{dfn:195}) requires some work. Moreover, these 
definitions agree with those in Section \ref{sec:idealistic} precisely when 
{\em weak proregularity} holds. 

On the other hand, the definitions in this section have the distinct advantage 
of being rather easy to manipulate, and their ``adjunction'',  
``universality'' and ``idempotence'' features are not hard to verify, 
as explained in Remarks \ref{rem:216}, \ref{rem:217} and \ref{rem:251}.

\begin{dfn} \label{dfn:195}
Let $\aa = (a_1, \ldots, a_n)$ be a finite sequence in the ring $A$. 
The {\em sequential derived $\aa$-torsion functor}
is the functor 
\[ \opn{K}^{\vee}_{\infty}(A; \aa) \ot_A (-) : \cat{D}(A) \to \cat{D}(A) . \]
\end{dfn}

The definition makes sense: the complex 
$\opn{K}^{\vee}_{\infty}(A; \aa)$ K-flat, and hence tensoring with it respects 
quasi-isomorphisms.

A hint that this is related to torsion is this: let $\a$ be the ideal in $A$ 
generated by the sequence $\aa$. From formulas (\ref{eqn:191}) and 
(\ref{eqn:195}) it is clear that 
for an $A$-module $M$ there is a canonical isomorphism 
\begin{equation} \label{eqn:215}
\Ga_{\a}(M) \cong \opn{H}^0 \bigl( \opn{K}^{\vee}_{\infty}(A; \aa) \ot_A 
M \bigr) .
\end{equation}

Given a complex of $A$-modules $M$, let 
\begin{equation} \label{eqn:196} 
\si_{\aa, M}^{\mrm{R}} : 
\opn{K}^{\vee}_{\infty}(A; \aa) \ot_A M \to  M 
\end{equation}
be the morphism in $\cat{D}(A)$ arising from the augmentation homomorphism
$\opn{K}^{\vee}_{\infty}(A; \aa) \to A$ in (\ref{eqn:205}). 

\begin{dfn} \label{dfn:196} 
A complex $M \in \cat{D}(A)$ is called {\em derived $\aa$-torsion  
in the sequential sense} if the morphism  
$\si_{\aa, M}^{\mrm{R}}$ in (\ref{eqn:196}) is an isomorphism. 
\end{dfn}

\begin{dfn} \label{dfn:197} 
Let $\aa = (a_1, \ldots, a_n)$ be a finite sequence in the ring $A$. 
The {\em sequential derived $\aa$-adic completion functor}
is the functor 
\[ \opn{RHom}_{A} \bigl( \opn{K}^{\vee}_{\infty}(A; \aa), - \bigr) : 
\cat{D}(A) \to \cat{D}(A) . \]
\end{dfn}

The augmentation of $\opn{K}^{\vee}_{\infty}(A; \aa)$ induces, for every 
$M \in \cat{D}(A)$, a morphism 
\begin{equation} \label{eqn:200} 
\tau_{\aa, M}^{\mrm{L}} : M \to 
\opn{RHom}_{A} \bigl( \opn{K}^{\vee}_{\infty}(A; \aa), M \bigr) . 
\end{equation}

\begin{dfn} \label{dfn:201}  
A complex $M \in \cat{D}(A)$ is called {\em derived $\aa$-adically  
complete in the sequential sense} if the morphism  
$\tau_{\aa, M}^{\mrm{L}}$ in (\ref{eqn:200}) is an isomorphism. 
\end{dfn}

\begin{prop} \label{prop:200} 
Let $\aa = (a_1, \ldots, a_n)$ be a sequence in the ring $A$. 
The following are equivalent for $M \in \cat{D}(A)$.
\begin{itemize}
\rmitem{i} $M$ is derived $\aa$-adically complete in the sequential sense
(Definition \ref{dfn:201}). 

\rmitem{ii} The object 
$\opn{RHom}_A \bigl( \opn{C}(A; \aa), M \bigr)$ 
in $\cat{D}(A)$ is zero. 
\end{itemize}
\end{prop}

\begin{proof}
The exact sequence (\ref{eqn:40}) gives a distinguished triangle 
\[  \opn{C}(A; \aa)[-1] \to \opn{K}^{\vee}_{\infty}(A; \aa) \to A 
\xar{\, \triangle \, } \]
in $\cat{D}(A)$. Applying the functor 
$\opn{RHom}_A(-, M)$ to this triangle, we get a new distinguished triangle 
\[ M \xar{\tau_{\aa, M}^{\mrm{L}}}
\opn{RHom}_{A} \bigl( \opn{K}^{\vee}_{\infty}(A; \aa), M \bigr) \to 
\opn{RHom}_A \bigl( \opn{C}(A; \aa), M \bigr) \xar{\, \triangle \, } \]
in $\cat{D}(A)$. By a standard fact on distinguished triangles, the morphism 
$\tau_{\aa, M}^{\mrm{L}}$ is an isomorphism iff the object 
$\opn{RHom}_A \bigl( \opn{C}(A; \aa), M \bigr)$ is zero. 
\end{proof}

\begin{rem} \label{rem:215}
Condition (ii) in Proposition \ref{prop:200} was the definition of a {\em 
cohomologically complete complex} in \cite{KS2}, in the special case when 
$n = 1$, so that, writing $a := a_1$, we have 
$\opn{C}(A; \aa) = A[a^{-1}]$.
\end{rem}

Here is the sequential variant of Proposition \ref{prop:210}.

\begin{prop} \label{prop:215}
Let $\aa$ and $\bb$ be finite sequences of elements of $A$, and let 
$\a$ and $\b$ the ideals of $A$ generated by the sequences $\aa$ and $\bb$ 
respectively. Assume that $\sqrt{\a} = \sqrt{\b}$. Then there is a canonical 
isomorphism 
$\opn{K}^{\vee}_{\infty}(A; \aa) \cong \opn{K}^{\vee}_{\infty}(A; \bb)$
in $\cat{D}(A)$, which respects the augmentations to $A$.

Therefore the sequential derived $\aa$-torsion and $\aa$-adic completion 
functors are isomorphic to the  sequential derived $\bb$-torsion and 
$\bb$-adic completion functors, respectively. 
\end{prop}

\begin{proof}
In \cite[Theorem 6.1]{PSY} it is proved that the complexes 
$\opn{Tel}(A; \aa)$ and $\opn{Tel}(A; \bb)$
are homotopy equivalent. An inspection of the proof of 
\cite[Theorem 6.1]{PSY} shows that this homotopy equivalence is canonical, and 
it respects the augmentations of the telescope complexes to $A$ (up to 
homotopy). The quasi-isomorphism (\ref{eqn:166}) lets us translate these facts 
to isomorphisms in $\cat{D}(A)$. 
\end{proof}

\begin{rem} \label{rem:216}
The sequential derived functors are insensitive to the ring $A$, or are 
universal, in the following sense. Consider derived completion. Since there is 
the canonical quasi-isomorphism (\ref{eqn:166}), and since $\opn{Tel}(A; \aa)$ 
is a K-projective complex, the sequential derived completion of a complex $M$ 
is 
$\opn{Hom}_A \bigl( \opn{Tel}(A; \aa), M \bigr) \in \cat{D}(A)$.
However, the telescope complex is defined over $\Z$, as shown in the last 
formula in (\ref{eqn:164}). This means that when we pass to the derived 
category $\cat{D}(\Z[\tt])$ by the restriction functor, we get 
\[ \opn{Hom}_A \bigl( \opn{Tel}(A; \aa), M \bigr) \cong 
\opn{Hom}_{\Z[\tt]} \bigl( \opn{Tel}(\Z[\tt]; \tt), M \bigr) \in 
\cat{D}(\Z[\tt]) . \]
\end{rem}

\begin{rem} \label{rem:217}
The sequential derived completion and torsion functors are adjoint to each 
other, for a trivial reason. This is an easy version of {\em Greenlees-May 
Duality}. 

Once more, we present these functors using the 
telescope complex $T := \opn{Tel}(A; \aa)$, which is a bounded complex of 
free $A$-modules, canonically quasi-isomorphic to 
$\opn{K}^{\vee}_{\infty}(A; \aa)$. 
Then the sequential derived $\aa$-adic completion of a complex $M$ is 
$\opn{Hom}_A(T, M)$, and the sequential derived $\aa$-torsion of $M$ is 
$T \ot_A M$.  

An easy calculation using Hom-tensor adjunction shows that 
for for arbitrary $M, N \in \cat{D}(A)$ there is a 
canonical isomorphism 
\[ \opn{Hom}_{\cat{D}(A)} (T \ot_A M, N) \cong 
\opn{Hom}_{\cat{D}(A)}(M, \opn{Hom}_A(T,  N)) . \]
\end{rem}

\begin{rem} \label{rem:251}
Another nice property of the sequential derived completion functor is that it 
is always {\em idempotent}. Using the notation $T := \opn{Tel}(A; \aa)$
of the previous remark, the augmentation homomorphism $T \to A$ 
induces two homomorphisms of complexes 
$T \ot_A T \to T$; and according to \cite[Lemma 7.9]{PSY} these are homotopy 
equivalences. Therefore for every $M \in \cat{D}(A)$ the two morphisms 
\[ \opn{Hom}_A(T, M) \to \opn{Hom}_A(T, \opn{Hom}_A(T, M)) \]
in $\cat{D}(A)$ are isomorphisms.
\end{rem}

\begin{rem} \label{rem:220}
Positselski has a much deeper understanding of various aspects of completion 
and  derived completion, including some intermediate operations between what we 
call the idealistic and the sequential derived completions. These are 
summarized in his paper \cite{Po}, where precise references can be found.  

For instance, fix a finite sequence $\aa = (a_1, \ldots, a_n)$ in $A$.
An $A$-module $M$ that is derived $\aa$-adically complete as a 
complex (in the sequential sense) is called an {\em $\aa$-contramodule}. 
For $M \in \cat{M}(A)$ to be an $\aa$-contramodule it is necessary and 
sufficient that $\opn{Ext}^q_A(A[a_i^{-1}], M) = 0$ for all 
$q = 0, 1$ and $i = 1, \ldots, n$. 

Let $\cat{M}_{\aa \tup{-ctra}}(A)$
be the full subcategory of $\cat{M}(A)$ on the $\aa$-contramodules. 
Positselski proves that $\cat{M}_{\aa \tup{-ctra}}(A)$
is a full abelian subcategory, closed under extensions. 
Another feature of contramodules is this: a complex $M \in \cat{D}(A)$ is 
derived $\aa$-adically complete (in the sequential sense) iff all its 
cohomology modules $\opn{H}^q(M)$ are $\aa$-contramodules. 

Here is one more consequence of Positselski's methods. Let $\a$ be the ideal 
generated by $\aa$. He proves that every complex $M$ that is derived 
$\a$-adically complete in the idealistic sense is also derived 
$\aa$-adically complete in the sequential sense; but not vice versa. See 
\cite[Lemma 5.1 and Proposition 5.3]{Po}.
\end{rem}

\section{Weak Proregularity: When Idealistic and Sequential Agree}
 
We continue with the commutative ring $A$.
In this section we recall the definition of {\em weak
proregularity}, and explain some of its useful properties. 

An inverse system of modules $\{ N_i \}_{i \in \N}$ is called
{\em pro-zero} if for every $i$ there is some $j \geq i$ such that the 
homomorphism $N_j \to N_i$ is zero. See Remark \ref{rem:300} regarding this 
notion. 

Given a sequence $\aa = (a_1, \ldots, a_n)$ in $A$, and 
a natural number $i$, we let 
$\bsym{a}^i := (a^i_1, \ldots, a^i_n)$.
As explained in Section \ref{sec:sequential}, the collection of Koszul 
complexes 
$\bigl\{ \opn{K}(A; \bsym{a}^i)) \bigr\}_{i \in \N}$
is an inverse system. 

\begin{dfn} \label{dfn:45}
A finite sequence $\bsym{a}$ in the ring $A$ is called {\em weakly proregular} 
(WPR) if for every $q < 0$ the inverse system of $A$-modules 
$\bigl\{  \opn{H}^q( \opn{K}(A; \bsym{a}^i)) \bigr\}_{i \in \N}$
is pro-zero.
\end{dfn}

The condition in Definition \ref{dfn:45} had already appeared in \cite{LC}, but 
the name was given much later, by Lipman, see \cite[Correction]{AJL}. 

\begin{dfn} \label{dfn:115}
An ideal $\a \sub A$ is called a {\em weakly proregular ideal} 
if it is generated by some weakly proregular sequence $\aa$.
\end{dfn}

In particular, a WPR ideal $\a$ is finitely generated. 
The next fact was already observed by Grothendieck:

\begin{thm}[{\cite[Lemma 2.4]{LC}}] \label{thm:10}
If the ring $A$ is noetherian, then every finite sequence in it is WPR. 
\end{thm}

The moral is that weak proregularity of an ideal is a generalization of 
the noetherian property of the ring. 
The next results show that WPR is a robust property. 

\begin{thm}[{\cite[Corollary 6.2]{PSY}}] \label{thm:141}
Let $\aa$ and $\bb$ be finite sequences of elements of $A$,
and let $\a$ and $\b$ be the ideals generated by $\aa$ and $\bb$ respectively. 
Assume that $\sqrt{\a} = \sqrt{\b}$. Then the sequence $\aa$ is WPR iff the 
sequence $\bb$ is WPR.
\end{thm}

\begin{cor}[{\cite[Corollary 6.3]{PSY}}] \label{cor:140}
Let $\a$ be a WPR ideal in $A$, and let $\aa$ be a finite sequence of 
elements that generates $\a$. Then $\aa$ is a WPR sequence. 
\end{cor}

\begin{prop} \label{prop:310}
Let $g : A \to B$ be a flat ring homomorphism, and let 
$\aa = (a_1, \ldots, a_n)$ be a WPR sequence in $A$.
Then the sequence 
$\bb := (g(a_1), \ldots, g(a_n))$ in $B$ is WPR. 
\end{prop}

\begin{proof}
For every $i$ there is an isomorphism of complexes 
$\opn{K}(B; \bb^i) \cong B \ot_A \opn{K}(A; \aa^i)$; cf.\ equation 
(\ref{eqn:164}). Thus for every $q$ and $i$ we get canonical isomorphisms
\[ \opn{H}^q( \opn{K}(A; \bb^i)) \cong 
\opn{H}^q( B \ot_A \opn{K}(A; \aa^i)) \cong^{(\dag)} 
B \ot_A \opn{H}^q( \opn{K}(A; \aa^i)) , \]
where $\cong^{(\dag)}$ is due to the flatness of $g$.
Fixing $q < 0$, these are isomorphisms of inverse systems indexed by 
$i \in \N$. Therefore inverse system of $B$-modules 
$\bigl\{  \opn{H}^q( \opn{K}(B; \bsym{b}^i)) \bigr\}_{i \in \N}$
is pro-zero.
\end{proof}

There is also a categorical characterization of the WPR property. 
Following \cite{VY} we make the next definition. Its origins can be traced back 
to  texts on abstract torsion classes. 

\begin{dfn} \label{dfn:215}
Let $F : \cat{M}(A) \to \cat{M}(A)$ be an additive functor. 
\begin{enumerate}
\item An $A$-module $I$ is called {\em right $F$-acyclic} if 
$\mrm{R}^q F(I) = 0$ for all $q > 0$. 

\item The functor $F$ is called {\em stable} if for every injective $A$-module 
$I$, the module $F(I)$ is injective. 

\item The functor $F$ is called {\em weakly stable} if for every injective 
$A$-module $I$, the module $F(I)$ is  right $F$-acyclic. 
\end{enumerate}
\end{dfn}

\begin{thm}[{\cite[Theorem 0.3]{VY}}]  \label{thm:216}
Let $\a$ be a finitely generated ideal in the ring $A$. The conditions below 
are equivalent. 
\begin{itemize}
\rmitem{i} The ideal $\a$ is weakly proregular.

\rmitem{ii} The torsion functor $\Ga_{\a}$ is weakly stable.
\end{itemize}
\end{thm}

The next two propositions describe the relationship between the idealistic and 
sequential derived completion and torsion functors. These results are implicit 
in \cite[Sections 4-5]{PSY}.

\begin{prop} \label{prop:315}
Let $\aa$ be a finite sequence of elements of $A$, and let $\a \sub A$
be the ideal generated by $\aa$. For every complex $M \in \cat{D}(A)$
there is a morphism
\[ v^{\mrm{R}}_{\aa, M} : \mrm{R} \Ga_{\a}(M) \to 
\opn{K}^{\vee}_{\infty}(A; \aa) \ot_A M \]
in $\cat{D}(A)$, which makes the diagram
\[ \UseTips \xymatrix @C=8ex @R=6ex {
\mrm{R} \Ga_{\a}(M)
\ar[r]^(0.4){ v^{\mrm{R}}_{\aa, M} }_(0.4){}
\ar[dr]_{ \si^{\mrm{R}}_{\a, M} }
&
\opn{K}^{\vee}_{\infty}(A; \aa) \ot_A M 
\ar[d]^{ \si^{\mrm{R}}_{\aa, M} }
\\
&
M
} \]
in $\cat{D}(A)$ commutative. This diagram is functorial in $M$. 
\end{prop}

\begin{proof}
This is part of the proof of \cite[Corollary 4.26]{PSY}, so we only recall the 
key points. It is enough to consider a K-injective complex $M = I \in 
\cat{D}(A)$. Then $v^{\mrm{R}}_{\aa, I}$ is represented by the homomorphism of 
complexes
\[  v_{\aa, I} : 
\Ga_{\a}(I) \to \opn{K}^{\vee}_{\infty}(A; \aa) \ot_A I \]
from \cite[Equation (4.19)]{PSY}.
The commutativity of the diagram is explained in loc.\ cit.
\end{proof}

\begin{prop} \label{prop:316}
Let $\aa$ be a finite sequence of elements of $A$, and let $\a \sub A$
be the ideal generated by $\aa$. For every complex $M \in \cat{D}(A)$
there is a morphism
\[ u^{\mrm{L}}_{\aa, M} : 
\opn{RHom}_{A} \bigl( \opn{K}^{\vee}_{\infty}(A; \aa), M \bigr)
\to \mrm{L} \La_{\a}(M) \]
in $\cat{D}(A)$, which makes the diagram
\[ \UseTips \xymatrix @C=8ex @R=6ex {
M
\ar[dr]^{ \tau^{\mrm{L}}_{\a, M} }
\ar[d]_{ \tau^{\mrm{L}}_{\aa, M} }
\\
\opn{RHom}_{A} \bigl( \opn{K}^{\vee}_{\infty}(A; \aa), M \bigr)
\ar[r]_(0.6){ u^{\mrm{L}}_{\aa, M}}^(0.6){}
&
\mrm{L} \La_{\a}(M)
} \]
in $\cat{D}(A)$ commutative. This diagram is functorial in $M$. 
\end{prop}

\begin{proof}
This is implicit in \cite[Section 5]{PSY}. It is enough to consider a 
K-projective complex $M = P \in \cat{D}(A)$. Recall that the telescope complex 
$T := \opn{Tel}(A; \aa)$ is a K-projective complex. Then 
$u^{\mrm{L}}_{\aa, P}$ is the composition of the isomorphism 
\[ \opn{RHom}_{A}(w_{\aa}, \opn{id}_P) : 
\opn{RHom}_{A} \bigl( \opn{K}^{\vee}_{\infty}(A; \aa), P \bigr) \iso
\opn{RHom}_{A}(T,  P) \cong  \opn{Hom}_{A}(T, P) , \]
where $w_{\aa}$ is the quasi-isomorphism from equation (\ref{eqn:166}), with 
the homomorphism of complexes
\[ \opn{tel}_{\lsp \aa, P} : \opn{Hom}_{A}(T, P) \to \La_{\a}(P)
\cong \mrm{L} \La_{\a}(P) \]
from \cite[Definition 5.16]{PSY}. 
The commutativity of the diagram is proved in loc.\ cit.
\end{proof}

\begin{thm} \label{thm:165} 
Let $\aa$ be a finite sequence in $A$, and let $\a \sub A$ be the ideal 
generated by $\aa$. The following conditions are equivalent. 
\begin{itemize}
\rmitem{i} The sequence $\aa$ is weakly proregular. 

\rmitem{ii} For every complex $M \in \cat{D}(A)$ the morphism
\[ v^{\mrm{R}}_{\aa, M} : \mrm{R} \Ga_{\a}(M) \to 
\opn{K}^{\vee}_{\infty}(A; \aa) \ot_A M  \]
from Proposition \ref{prop:315} is an isomorphism.

\rmitem{iii} For every complex $M \in \cat{D}(A)$ the morphism
\[ u^{\mrm{L}}_{\aa, M} : 
\opn{RHom}_{A} \bigl( \opn{K}^{\vee}_{\infty}(A; \aa), M \bigr)
\to \mrm{L} \La_{\a}(M) \]
from Proposition \ref{prop:316} is an isomorphism.
\end{itemize}
\end{thm}

\begin{proof}
The implication (i) $\Rightarrow$ (ii) is \cite[Corollary 4.26]{PSY}.
The implication (i) $\Rightarrow$ (iii) is \cite[Corollary 5.25]{PSY},
combined with the quasi-isomorphism (\ref{eqn:166}). 

The implication (ii) $\Rightarrow$ (i) is \cite[Theorem 4.24]{PSY}, applied to 
$M := I$, an arbitrary injective $A$-module. 

Finally, the implication (iii) $\Rightarrow$ (i) is a recent result of 
Positselski. One takes $M := P$, a free $A$-module of infinite rank.
Then $\mrm{L} \La_{\a}(P) = \La_{\a}(P)$. 
According to \cite[Theorem 3.6]{Po}, if $u^{\mrm{L}}_{\aa, P}$ is an 
isomorphism in $\cat{D}(A)$, then $\a$ is WPR. 
\end{proof}

\begin{cor} \label{cor:152}
Assume that $\a \sub A$ is a WPR ideal, and $\aa$ is some finite sequence 
that generates $\a$. Let $M \in \cat{D}(A)$. The following two conditions 
are equivalent. 
\begin{itemize}
\rmitem{i} $M$ is  derived $\a$-adically complete in the idealistic sense 
(Definition \ref{dfn:181}). 

\rmitem{ii} $M$ is  derived $\aa$-adically complete in the sequential sense 
(Definition \ref{dfn:201}). 
\end{itemize}
\end{cor}

\begin{proof}
Combine Corollary \ref{cor:140}, the theorem above, and the commutativity of 
the diagram in Proposition \ref{prop:316}.
\end{proof}

There is a corresponding result for derived torsion, proved using 
Proposition \ref{prop:315}. 

Given a finitely generated ideal $\a \sub A$, we denote by 
$\cat{D}(A)_{\a \tup{-tor}}$ and $\cat{D}(A)_{\a \tup{-com}}$
the full subcategories of $\cat{D}(A)$ on the complexes that are idealistically
derived $\a$-torsion and idealistically derived $\a$-adically complete, 
respectively. These are triangulated subcategories. 

\begin{thm}[MGM Equivalence, {\cite[Theorem 1.1]{PSY}}] \label{thm:220}
Let $\a$ be a weakly proregular ideal in the ring $A$. Then:
\begin{enumerate}
\item For every $M \in \cat{D}(A)$ one has
$\mrm{R} \Gamma_{\a} (M) \in \cat{D}(A)_{\a \tup{-tor}}$
and 
$\mrm{L} \Lambda_{\a} (M) \in \cat{D}(A)_{\a \tup{-com}}$.

\item The functor 
\[ \mrm{R} \Gamma_{\a} : 
\cat{D}(A)_{\a \tup{-com}} \to 
\cat{D}(A)_{\a \tup{-tor}} \]
is an equivalence, with quasi-inverse $\mrm{L} \Lambda_{\a}$. 
\end{enumerate}
\end{thm}

A noncommutative version of weak proregularity (condition (ii) in Theorem 
\ref{thm:216} above), and the corresponding noncommutative MGM equivalence, can 
be found in the paper \cite{VY}.

\begin{rem} \label{rem:300}
Let $\{ N_i \}_{i \in \N}$ be an inverse system in $\cat{M}(A)$. 
Often the condition of being pro-zero is called the 
{\em trivial Mittag-Leffler condition}.
In \cite{LC} this is called {\em essentially zero}. 

Recall  the category $\cat{Pro}(\cat{M}(A))$ of 
pro-objects of $\cat{M}(A)$. It is the full subcategory of 
$\cat{Fun}(\cat{M}(A), \cat{Set})^{\mrm{op}}$ on the filtered 
colimits, in $\cat{Fun}(\cat{M}(A), \cat{Set})$, of corepresentable functors.
The category $\cat{Pro}(\cat{M}(A))$ is abelian.
See \cite[Section 1.7]{Ye3}, \cite[Remark 1.8.8]{Ye3}, 
\cite[Section 1.11]{KS1}, \cite[Section 6.1]{KS2} and 
\cite[Section 8.6]{KS2}. 

Given an inverse system $\{ N_i \}_{i \in \N}$ in $\cat{M}(A)$, let
$F_i := \opn{Hom}_A(N_i, -) : \cat{M}(A) \to \cat{Set}$, 
and let  
$F := {''{\underset{\lto i}{\mrm{lim}}}''} \lsp N_i =  
\underset{i \to}{\mrm{lim}} \, F_i$
be the pro-object obtained from the direct system of functors
$ \{ F_i \}_{i \in \N}$. 
For every $M \in \cat{M}(A)$ we have 
$F(M) = \underset{i \to}{\mrm{lim}} \, F_i(M) = 
\underset{i \to}{\mrm{lim}} \, \opn{Hom}_A(N_i, M)$.
We see that the inverse system $\{ N_i \}_{i \in \N}$ is pro-zero iff the 
pro-object 
${''{\underset{\lto i}{\mrm{lim}}}''} \lsp N_i$ 
is the zero object in $\cat{Pro}(\cat{M}(A))$.
\end{rem}

\section{Weak Proregularity and Adic Flatness}

Again, $A$ is some commutative ring. 

\begin{dfn} \label{dfn:225}
Let $\a$ be a finitely generated ideal in $A$, and let $M$ be an $A$-module. We 
say that $M$ is {\em $\a$-adically flat} if $\opn{Tor}^{A}_q(N, M) = 0$
for every $\a$-torsion $A$-module $N$ and every $q > 0$. 
\end{dfn}

This definition is copied from \cite{Ye2}. In \cite{BS2}, \cite{CS} 
the term used is {\em $\a$-completely flat}.
Here is a useful characterization of this property. 

\begin{thm}[{\cite[Theorem 1.3]{Ye2}}] \label{thm:225}
Let $\a$ be a finitely generated ideal in $A$, and for every $k \geq 0$ let 
$A_k := A / \a^{k + 1}$. The following three conditions are equivalent for an 
$A$-module $M$.
\begin{enumerate}
\rmitem{i} The $A$-module $M$ is $\a$-adically flat.

\rmitem{ii} For every $q > 0$ and $k \geq 0$ the module 
$\opn{Tor}^{A}_q(A_{k}, M)$ vanishes, and   
$A_{k} \ot_A M$ is a flat $A_k$-module.

\rmitem{iii} For every $q > 0$ the module 
$\opn{Tor}^{A}_q(A_{0}, M)$ vanishes, and   
$A_{0} \ot_A M$ is a flat $A_0$-module.
\end{enumerate}
\end{thm}

The next theorem was considered by many to be unproved; but then, a few 
years ago, several different proofs of it have emerged. One of them -- the 
proof from \cite{Ye2} -- will be mentioned a bit later.

\begin{thm} \label{thm:226}
If $A$ is a noetherian commutative ring, $\a$ is an ideal in $A$, and $M$ is a 
flat $A$-module, then the $\a$-adic completion $\what{M} = \La_{\a}(M)$ is a 
flat $A$-module.
\end{thm}

Here is a similar result -- the assumptions are weaker and so is the outcome. 

\begin{thm}[{\cite[Theorem 1.4]{Ye2}}] \label{thm:227}
Let $\a$ be a weakly proregular ideal in $A$, and let $M$ be an $\a$-adically 
flat $A$-module, with $\a$-adic completion 
$\what{M} = \La_{\a}(M)$. Then the $A$-module $\what{M}$ is $\a$-adically flat. 
\end{thm}

When the ring $A$ is noetherian, the two notions of flatness coincide:

\begin{thm}[{\cite[Theorem 1.5]{Ye2}}] \label{thm:228}
If $A$ is a noetherian ring, $\a$ is an ideal in $A$, and 
$\what{M}$ is an $\a$-adically flat $\a$-adically complete $A$-module,
then $\what{M}$ is a flat $A$-module. 
\end{thm}

Theorem \ref{thm:228} was used in the recent preprint \cite{CS} -- see Lemma 
6.2.1 of op.\ cit. 

Here we feel an example is due, showing that the concepts we are talking about 
are truly distinct. 

\begin{exa} \label{exa:225}
Let $\K$ be a field of characteristic $0$, let $\K[[t_1]]$ and 
$\K[[t_2]]$ be the rings of power series in the variables $t_1$ and 
$t_2$, and let $A$ be the ring
\[ A := \K[[t_1]] \ot_{\K} \K[[t_2]] . \]
Let $\a$ be the ideal in $A$ generated by $t_1$ and $t_2$,
and let $\what{A}$ be the $\a$-adic completion of $A$. 
Then, according to \cite[Theorem 7.2]{Ye2}, the following hold:
\begin{enumerate}
\item The ideal $\a$ is weakly proregular. 
\item The ring $A$ is not noetherian.
\item The ring $\what{A}$ is noetherian.
\item The ring $\what{A}$ is $\a$-adically flat over $A$. 
\item The ring $\what{A}$ is not flat over $A$. 
\end{enumerate}
\end{exa}

Quite surprisingly, weak proregularity is a consequence of the preservation of 
adic flatness under completion. Indeed, we have the next result of Positselski, 
whose proof requires the deeper methods alluded to in Remark \ref{rem:220}.

\begin{thm}[{\cite[Theorem 7.2]{Po}}] \label{thm:230}
Let $\a$ be a finitely generated ideal in $A$, and let $P$ be a free $A$-module 
of infinite rank. If the $A$-module 
$\what{P} = \La_{\a}(P)$ is $\a$-adically flat, then the ideal $\a$ is weakly 
proregular. 
\end{thm}

\begin{rem} \label{rem:230}
There is a description of the free module $P$ and its completion $\what{P}$ 
from the theorem above, which is sometimes useful, and is inspired by 
functional analysis. Given a set $X$, the module of finitely supported functions 
$f : X \to A$ is denoted by $\opn{F}_{\mrm{fin}}(X, A)$.
This is a free $A$-module with basis the delta functions $\de_x$. 

Now take a finitely generated ideal $\a \sub A$, and let $\what{A}$ be the 
$\a$-adic completion of $A$.
A function $f : X \to \what{A}$ is called {\em $\a$-adically decaying} if for 
every $k \geq 1$, the set 
$\{ x \in X \mid f(x) \notin \a^{k} \cd \what{A} \}$ 
is finite. We denote by 
$\opn{F}_{\mrm{dec}}(X, \what{A})$ the $A$-module of decaying functions. 
It turns out that $\opn{F}_{\mrm{dec}}(X, \what{A})$ is the 
$\a$-adic completion of $\opn{F}_{\mrm{fin}}(X, A)$.
See \cite[Corollary 2.9]{Ye1}.
\end{rem}

The next concept we shall introduce had a pivotal role in our study of adic 
flatness in the paper \cite{Ye2}, and we hope some readers might also find it 
useful.  

\begin{dfn} \label{dfn:230}
Let $\a$ be a  finitely generated ideal in $A$, and for every $k \geq 0$ let 
$A_k := A / \a^{k + 1}$. 
An {\em $\a$-adic system of $A$-modules} 
is an inverse system $\{ M_k \}_{k \in \N}$
of $A$-modules, such that each $M_k$ is an $A_k$-module, and for each $k$ 
the induced homomorphism 
$A_k \ot_{A_{k + 1}} M_{k + 1} \to M_k$
is bijective. 
\end{dfn}

The notation of this definition will be used implicitly below. 

\begin{exa} \label{exa:240}
An $A$-module $M$ gives rise to an $\a$-adic system 
$\{ M_k \}_{k \in \N}$, where $M_k := A_k \ot_{A} M$. 
\end{exa}

\begin{thm}[{\cite[Theorem 1.2]{Ye2}}] \label{thm:231}
Let $\a$ be a finitely generated ideal in $A$, 
and let $\{ M_k \}_{k \in \N}$ be an $\a$-adic system of 
$A$-modules, with limit $\what{M} := \lim_{\leftarrow k} M_{k}$. Then:
\begin{enumerate}
\item The $A$-module $\what{M}$ is $\a$-adically complete.

\item For every $k \geq 0$ the canonical homomorphism 
$A_k \ot_A \what{M} \to M_k$ is bijective.
\end{enumerate} 
\end{thm}

\begin{dfn} \label{dfn:232}
An $\a$-adic system $\{ M_k \}_{k \in \N}$ is called {\em flat} if each $M_k$ 
is a flat $A_k$-module.  
\end{dfn}

\begin{thm}[{\cite[Theorem 1.6]{Ye2}}] \label{thm:233}
Let $\a$ be a  finitely generated ideal in $A$, and let $\{ M_k \}_{k \in \N}$ 
be a flat  $\a$-adic system, with limit $\what{M} := \lim_{\leftarrow k} M_{k}$.
\begin{enumerate} 
\item If the ideal $\a$ is weakly proregular, then $\what{M}$ is an 
$\a$-adically flat $A$-module. 
\item If the ring $A$ is noetherian, then $\what{M}$ is a flat $A$-module. 
\end{enumerate}
\end{thm}

Item (1) of this theorem, coupled with Theorem \ref{thm:225}, immediately 
implies Theorem \ref{thm:227}. Item (2) of the theorem implies Theorem 
\ref{thm:226}.

\section{Weak Proregularity and Quotient Rings}

Again $A$ is some commutative ring. The purpose of this extremely technical 
section is to prove Theorem \ref{thm:105}. 

\begin{dfn} \label{dfn:106} 
An element $a \in A$ is called a {\em weakly proregular element} 
if the length one sequence $\aa := (a)$ is weakly proregular in $A$.
\end{dfn}

Given an element $a \in A$ and an $A$-module $M$, the {\em annihilator} of $a$ 
in $M$ is the submodule 
\[ \opn{Ann}_M(a) := \opn{Ker} \bigl( \opn{mult}_M(a) \bigr) = 
\{ m \in M \mid a \cd m = 0 \} \sub  M . \]
The element $a$ is called {\em $M$-regular} or a {\em non-zero-divisor on $M$}
if $\opn{Ann}_M(a) = 0$. This means that $\opn{mult}_M(a)$ is an 
injective endomorphism of $M$. When $M = A$ we just call $a$ a {\em regular 
element}. Of course 
\begin{equation} \label{eqn:102}
\opn{Ann}_M(a^i) \sub \opn{Ann}_M(a^j) \sub M
\end{equation}
if $i \leq j$. For such $i, j$ there is a homomorphism 
\begin{equation} \label{eqn:130} 
\opn{mult}_M(a^{j - i}) : \opn{Ann}_M(a^j) \to \opn{Ann}_M(a^i) . 
\end{equation}
Note that $a^0 = 1$, so that $\opn{Ann}_M(a^{0}) = 0$. The element $a$ is 
$M$-regular iff  $\opn{Ann}_M(a^{}) = 0$, iff 
$\opn{Ann}_M(a^{i}) = 0$ for all $i \in \N$.

As can be seen immediately from (\ref{eqn:101}), for every $i \in \N$
there is equality 
\begin{equation} \label{eqn:105}
\opn{Ann}_M(a^i) = \opn{H}^{-1} \bigl( \opn{K}(A; a^{i}) \ot_A M \bigr)
\end{equation}
of submodules of $M$, after we make the obvious identification 
$M \cong A \ot_A M$.

\begin{dfn} \label{dfn:100}
Let $a \in A$ and let $M$ be an $A$-module. We say that $M$ has {\em bounded 
$a$-torsion} if there is some $j_0 \in \N$ such that 
$\opn{Ann}_M(a^{j}) = \opn{Ann}_M(a^{j_0})$ for all $j \geq j_0$.

Let $\opn{tb}_{M}(a)$ be the smallest such number $j_0 \in \N$, if it exists; 
and otherwise let $\opn{tb}_{M}(a) := \infty$. The generalized number
$\opn{tb}_{M}(a) \in \N \cup \{ \infty \}$ is called the {\em torsion bound} of
$a$ on $M$. 

\end{dfn}

Thus $M$ has bounded $a$-torsion iff $\opn{tb}_{M}(a) < \infty$.
Also $a$ is an $M$-regular element iff $\opn{tb}_{M}(a) = 0$. 

The next proposition is well-known, yet we find 
it instructive to give a proof here, since the same ideas will be used in the 
subsequent lemmas.

\begin{prop} \label{prop:100}
Let $a \in A$. The following conditions are equivalent\tup{:}
\begin{itemize}
\rmitem{i} The element $a$ is WPR. 

\rmitem{ii} $A$ has bounded $a$-torsion. 
\end{itemize}
\end{prop}

\begin{proof} \mbox{}

\smallskip \noindent 
(ii) $\Rightarrow$ (i): Let $i_0 := \opn{tb}_{A}(a) \in \N$. 
We must prove that given $i$, there 
exists some $j \geq i$ such that the homomorphism 
\[ \opn{mult}(a^{j - i}) : \opn{Ann}_A(a^j) \to \opn{Ann}_A(a^i) \]
is the zero homomorphism. This just means that $a^{j - i}$ annihilates 
$\opn{Ann}_A(a^j)$; or in other words, taking the inclusion (\ref{eqn:102}) 
into account, that 
$\opn{Ann}_M(a^j) = \opn{Ann}_M(a^{j - i})$.  
By the choice of $i_0$, the number $j := i + i_0$ works. 

\medskip \noindent 
(i) $\Rightarrow$ (ii): By the WPR condition, with $i = 1$, there is some 
$j_0 \in \N$ such that the homomorphism 
\begin{equation} \label{eqn:106}
\opn{mult}(a^{j_0}) : \opn{Ann}_A(a^{j_0 + 1}) \to \opn{Ann}_A(a^1)
\end{equation}
is the zero homomorphism. We will prove that 
$\opn{tb}_{A}(a) \leq j_0$. This will be done as follows: by induction on 
$j \geq j_0$, we will prove that 
$\opn{Ann}_A(a^{j}) = \opn{Ann}_A(a^{j + 1})$. 

\medskip \noindent 
$\triangleright$ \, 
For $j = j_0$ this is simply the vanishing of the homomorphism (\ref{eqn:106}). 

\medskip \noindent 
$\triangleright$ \, Assume this is true for $j \geq j_0$. Consider an 
element $b \in \opn{Ann}_A(a^{j + 2})$. Then 
$a \cd b \in \opn{Ann}_A(a^{j + 1})$. By assumption 
$a \cd b \in \opn{Ann}_A(a^{j})$. Hence $b \in \opn{Ann}_A(a^{j + 1})$.
\end{proof}

\begin{lem} \label{lem:320} 
Let $b \in A$, and let 
\[ 0 \to M' \xar{\ep} M \xar{\pi} M'' \to 0 \]
be a short exact sequence in $\cat{M}(A)$. Then 
$\opn{tb}_{M}(b) \leq \opn{tb}_{M'}(b) + \opn{tb}_{M''}(b)$.
\end{lem}

\begin{proof}
Let's write $l' := \opn{tb}_{M'}(b)$ and $l'' := \opn{tb}_{M''}(b)$; and we may 
assume that $l', l'' < \infty$. 
Given some $j \geq l' + l''$ and some $m \in M$ s.t.\
$b^j \cd m = 0$, we must prove that 
$b^{l' + l''} \cd m = 0$. 

Since $b^j \cd m = 0$ in $M$, it follows that 
$b^j \cd \pi(m) = 0$ in $M''$. Now $j \geq l''$, so by the definition of 
$l'' = \opn{tb}_{M''}(b)$ we have $b^{l''} \cd \pi(m) = 0$ in $M''$.

Define $n := b^{l''} \cd m \in M$. Then $\pi(n) = 0$, i.e.\ 
$n \in \opn{Ker}(\pi)$. Let us identify $\opn{Ker}(\pi) \sub M$ with $M'$ by 
way of $\ep$. We know that $b^j \cd n = b^{j + l''} \cd m = 0$, and that $j 
\geq l'$; so by the definition of $l' = \opn{tb}_{M'}(b)$ it follows that 
$b^{l'} \cd n = 0$ in $M' \sub M$. 

Combining the calculations above we obtain 
$b^{l' + l''} \cd m = b^{l'} \cd n = 0$,
as required. 
\end{proof}

\begin{lem} \label{lem:105}
Let $a, b \in A$. For every $k \in \N$ let $A_k := A / (a^{k + 1})$.
Assume that $a$ is a regular element, and that $A_0$ has bounded $b$-torsion. 
Then $A_k$ has bounded $b$-torsion for every $k$. 
\end{lem}

\begin{proof}
The assumption is that $l := \opn{tb}_{A_0}(b) < \infty$. 
We shall prove that 
\begin{equation} \label{eqn:120}
\opn{tb}_{A_k}(b) \leq (k + 1) \cd l 
\end{equation}
for all $k \geq 0$.
 
Define the ideal $\a := (a) \sub A$.
For  $k \geq 1$ define the $A$-module
\[ N_k := \a^k / \a^{k + 1} \cong \opn{Ker}(A_k \xar{\pi_k} A_{k - 1}) , \]
where $\pi_k : A_k \xar{} A_{k - 1}$ is the $A$-ring homomorphism. 
And let $N_0 := A_0$. 
Since $a$ is a regular element of $A$, it follows that 
$\opn{mult}_A(a) : \a^k \to \a^{k + 1}$
is bijective for all $k \geq 0$; and hence it induces an 
$A$-module isomorphism $N_k \iso N_{k + 1}$. We see that  
$\opn{tb}_{N_k}(b) = l$ for all $k \geq 0 $. 

We shall now prove (\ref{eqn:120}), by induction on $k$.
For $k = 0$ there is nothing to prove.
Next, take any $k \geq 0$, and assume formula (\ref{eqn:120}) holds for $k$. 
Consider the short exact sequence of $A$-modules 
\begin{equation} \label{eqn:121}
0 \to N_{k + 1} \to  A_{k + 1} \xar{\pi_{k + 1}} A_k \to 0 . 
\end{equation}
According to Lemma \ref{lem:320}, applied to this short exact sequence, we 
obtain 
\[ \opn{tb}_{A_{k + 1}}(b) \leq 
\opn{tb}_{N_{k + 1}}(b) + \opn{tb}_{A_{k}}(b) \leq 
l + (k + 1) \cd l = (k + 2) \cd l . \]
This is the inequality (\ref{eqn:120}) for $k + 1$.
\end{proof}

\begin{thm} \label{thm:105}
Let $A$ be a commutative ring, and let $a, b \in A$. Assume that $a$ is 
a regular element of $A$, and the image $\bar{b}$ of $b$ in 
$\mbar{A} := A / (a)$ is a weakly proregular element. Then the length $2$ 
sequence $\aa := (a, b)$ in $A$ is weakly proregular.
\end{thm}

\begin{proof}
The Koszul complexes $\opn{K}(A; \aa^i)$ are concentrated in degrees 
$0, -1, -2$. We need to prove that for $q = -1, -2$ and for  
$i \geq 0$ there exists some $j \geq i$ for which the homomorphism 
\begin{equation} \label{eqn:133}
\opn{H}^q(\mu_{j, i}) : \opn{H}^q \bigl( \opn{K}(A; \aa^j) \bigr) \to
\opn{H}^q \bigl( \opn{K}(A; \aa^i) \bigr) 
\end{equation}
is zero. We will consider all these cases of $q$ and $i$ in three steps. 
We shall use the fact that 
\begin{equation} \label{eqn:134}
\opn{K}(A; \aa^i) = \opn{K}(A; a^i) \ot_A \opn{K}(A; b^i)
\end{equation}
as complexes.

\medskip \noindent 
Step 1. Here we deal with $i = 0$. 
Since $a^0 = b^0 = 1$, the complexes 
$\opn{K}(A; a^0)$ and $\opn{K}(A; b^0)$ are acyclic. 
From (\ref{eqn:134}) we conclude that the 
complex $\opn{K}(A; \aa^0)$ is acyclic. Hence the homomorphism 
(\ref{eqn:133}) is zero for all $q$ and all $j \geq 0$. 

\medskip \noindent 
Step 2. Here we deal with $q = -2$. Take any $i \geq 1$. 
The regularity of $a$ implies that 
$\opn{H}^{-1} \bigl( \opn{K}(A; a^i) \bigr) = 0$, 
so the augmentation homomorphism
\begin{equation} \label{eqn:123}
\opn{K}(A; a^i) \to \opn{H}^{0} \bigl( \opn{K}(A; a^i) \bigr)
= A / (a^i) = A_{i - 1} 
\end{equation}
is a quasi-isomorphism. From (\ref{eqn:134}), with the fact that 
$\opn{K}(A; b^i)$ is a K-flat complex, we see that the homomorphism of 
complexes 
\begin{equation} \label{eqn:240}
\opn{K}(A; \aa^i) \to A_{i - 1} \ot_A \opn{K}(A; b^i)
\end{equation}
induced by (\ref{eqn:123}) is a quasi-isomorphism. But 
\begin{equation} \label{eqn:241}
A_{i - 1} \ot_A \opn{K}(A; b^i) \cong \opn{K}(A_{i - 1}; b^i)
\end{equation}
canonically as complexes of $A$-modules.
The complex $\opn{K}(A_{i - 1}; b^i)$ is concentrated in degrees $-1, 0$. 
It follows that 
$\opn{H}^{-2} \bigl( \opn{K}(A; \aa^i) \bigr) = 0$,
and hence the homomorphism (\ref{eqn:133}) is zero for $q = -2$ and all 
$i \geq 1$. 

\medskip \noindent 
Step 3. Here we handle the case $q = -1$ and $i \geq 1$. For every 
$j \geq i \geq 1$ the ring homomorphism 
$\pi : A_{j - 1} \to A_{i - 1}$
induces a homomorphism of complexes 
\[  \opn{K}(\pi; b^j) : \opn{K}(A_{j - 1}; b^j) \to 
\opn{K}(A_{i - 1}; b^j) . \]
Taking the $q = -1$ cohomology in (\ref{eqn:240}) and (\ref{eqn:241}), we 
obtain a canonical isomorphism of $A$-modules 
\[ \opn{H}^{-1} \bigl( \opn{K}(A; \aa^j) \bigr) \cong 
\opn{H}^{-1} \bigl( \opn{K}(A_{j - 1}; b^j) \bigr) . \]
Likewise with $i$ instead of $j$. These, with the isomorphism 
(\ref{eqn:105}), make the diagram 
\begin{equation} \label{eqn:124}
\UseTips \xymatrix @C=6ex @R=6ex {
\opn{H}^{-1} \bigl( \opn{K}(A; \aa^j) \bigr)
\ar[d]_{ \opn{H}^{-1}(\mu_{j, i}) }
\ar[r]^(0.45){\simeq}
&
\opn{H}^{-1} \bigl( \opn{K}(A_{j - 1}; b^j) \bigr) 
\ar[d]_{ \opn{H}^{-1}(\mu_{j, i} \, \circ \, \opn{K}(\pi; b^j)) }
\ar[r]^(0.55){\simeq}
&
\opn{Ann}_{A_{j - 1}}(b^j)
\ar[d]_{ \opn{mult}(b^{j - i}) \, \circ \, \pi }
\\
\opn{H}^{-1} \bigl( \opn{K}(A; \aa^i) \bigr)
\ar[r]^(0.45){\simeq}
&
\opn{H}^{-1} \bigl( \opn{K}(A_{i - 1}; b^i) \bigr) 
\ar[r]^(0.55){\simeq}
&
\opn{Ann}_{A_{i - 1}}(b^i)
} 
\end{equation}
commutative. The rightmost column factors into this commutative diagram:
\begin{equation} \label{eqn:127}
\UseTips \xymatrix @C=6ex @R=6ex {
\opn{Ann}_{A_{j - 1}}(b^j)
\ar[r]^{ \pi }
\ar[dr]_{ \opn{mult}(b^{j - i}) \, \circ \, \pi \ }
&
\opn{Ann}_{A_{i - 1}}(b^j)
\ar[d]^{ \opn{mult}_{A_{i - 1}}(b^{j - i}) }
\\
&
\opn{Ann}_{A_{i - 1}}(b^i)
}
\end{equation}

Let's fix $i \geq 1$ now. By Lemma \ref{lem:105} the $b$-torsion on 
$A_{i - 1}$ is bounded, i.e.\ 
$\opn{tb}_{A_{i - 1}}(b) < \infty$. This implies that for every 
$j \geq i + \opn{tb}_{A_{i - 1}}(b)$
the vertical arrow in (\ref{eqn:127}) is zero. Going back to 
diagram (\ref{eqn:124}), we see that the leftmost vertical arrow in it is zero 
for every $j \geq i + \opn{tb}_{A_{i - 1}}(b)$.
The conclusion is that the homomorphism (\ref{eqn:133}) is zero 
for $q = -1$ and $j \gg i$.
\end{proof}

\begin{rem} \label{rem:320}
In an earlier version of this paper we asked the following three questions, 
which are variations of Theorem \ref{thm:105}.
Let $\a$ and $\b$ be finitely generated ideals in the ring $A$. Define the ring 
$\mbar{A} := A / \a$ and the ideal $\bar{\b} := \b \cd \mbar{A} \sub \mbar{A}$.
\begin{enumerate}
\item[(Q1)] Suppose $\a$ and $\b$ are WPR. Is the ideal $\a + \b \sub A$ WPR?

\item[(Q2)] Suppose $\a$ and $\b$ are WPR. Is the ideal 
$\bar{\b} \sub \mbar{A}$ WPR?

\item[(Q3)] (Disegni) Assume that $\a \sub A$ and $\bar{\b} \sub \mbar{A}$ are 
WPR ideals. Is the ideal $\a + \b \sub A$ WPR?
\end{enumerate}

Since then, L. Positselski (in a private communication) provided us with 
counterexamples to the first two questions, which we now reproduce. 
Consider a ring $C$, with an element $c \in C$ that is not weakly proregular 
(there are known examples of this in the literature). 
By Corollary \ref{cor:140}, the ideal 
$\c := (c) \sub C$ is not WPR. Let 
$A := C[t]$, the polynomial ring in the variable $t$, so that 
$A \cong C \ot_{\Z} D$, where $D := \Z[t]$. Define the element
$b := c + t \in A$, and the principal ideals $\a := (t)$ and $\b := (b)$ in 
$A$. The elements $t, b \in A$ are regular, and hence the ideals $\a, \b \sub A$ 
are WPR. 

The ring $\bar{A} = A / \a$ is isomorphic to $C$, and the ideal 
$\bar{\b} = \b \cd \bar{A} \sub \bar{A}$ is sent to the ideal 
$\c \sub C$ under this isomorphism. By assumption the ideal 
$\c \sub C$ is not WPR. This is a counterexample to (Q2). 

Next let's consider the ideal 
$\a + \b \sub A$. It is easy to see that the ideal $\a + \b$ is generated by 
the sequence $(t, c)$. We will prove that the sequence $(t, c)$ in $A$ is not 
WPR, implying that the ideal $\a + \b$ is not WPR, thus providing a 
counterexample to (Q1). 

For $k \geq 0$ define 
$A_k := A / \a^{k + 1} = A / (t ^{k + 1})$
and $D_k :=  D / (t ^{k + 1}) = \Z[t] / (t ^{k + 1})$. There is a canonical 
ring isomorphism 
$A_k \cong C \ot_{\Z} D_k$, and a canonical isomorphism of complexes
\begin{equation} \label{eqn:322}
\opn{K}(A; t^{k + 1}, c^{k + 1}) \cong 
\opn{K}(D; t^{k + 1}) \ot_{\Z} \opn{K}(C; c^{k + 1}) .
\end{equation}
The canonical homomorphism $\opn{K}(D; t^{k + 1}) \to D_k$ is a 
quasi-isomorphism. Due to the flatness of $D$ and $D_k$ over $\Z$ we get a 
canonical isomorphism of $C$-modules 
\begin{equation} \label{eqn:321}
\opn{H}^{-1} \bigl( \opn{K}(D; t^{k + 1}) \ot_{\Z} \opn{K}(C; c^{k + 1}) 
\bigr) \cong 
D_k \ot_{\Z} \opn{H}^{-1} \bigl( \opn{K}(C; c^{k + 1}) \bigr) .
\end{equation}
Because the homomorphisms $D_{k + 1} \to D_k$ are surjective, and the 
inverse system
\[ \bigl\{ \opn{H}^{-1} \bigl( \opn{K}(C; c^{k + 1}) \bigr) \bigr\}_{k \in \N} 
\]
is not pro-zero, it follows that the inverse system (\ref{eqn:321}) is not 
pro-zero. Combining this with the canonical isomorphisms (\ref{eqn:322}) we 
conclude that the sequence $(t, c)$ in $A$ is not WPR. 

As for (Q3): it is still open. 
\end{rem}

\section{The Local Nature of Weak Proregularity}

As before, $A$ is a commutative ring. 
A sequence of elements $\ss = (s_1, \ldots, s_n)$ in $A$ is called a {\em 
covering sequence} if
$\opn{Spec}(A) = \bigcup\nolimits_{i} \opn{Spec}(A_{s_i})$.
Here $A_{s_i} = A[s_i^{-1}]$, the localized ring. 
Clearly $\ss$ is a covering sequence iff 
$\sum_{i = 1, \ldots, n} A \cd s_i = A$.

\begin{thm} \label{thm:306} 
Let $A$ be a ring and let $\a \sub A$ be an ideal.
The following two conditions are equivalent\tup{:}
\begin{itemize}
\rmitem{i} The ideal $\a$ is weakly proregular.  

\rmitem{ii} There is a covering sequence $\ss = (s_1, \ldots, s_n)$  of $A$, 
such that for every $i$ the ideal 
$\a_{s_i} := A_{s_i} \ot_A \a \sub A_{s_i}$
is weakly proregular. 
\end{itemize}
\end{thm}

\begin{proof} \mbox{} 

\smallskip \noindent 
(i) $\Rightarrow$ (ii): This is trivial, since we can take the covering sequence 
$\ss$ of $A$ with $n = 1$ and $s_1 = 1 \in A$.

\medskip \noindent
(ii) $\Rightarrow$ (i): 
For $i = 1, \ldots, n$ let $g_i : A \to A_{s_i}$ be the localization ring 
homomorphism. For every $i$ there is some WPR sequence 
$\aa_i$ in $A_{s_i}$ that generates the ideal $\a_{s_i}$.
Now the element $a_{i,j} \in A_{s_i}$, the $j$-th element in the sequence 
$\aa_i$, is of the form $a_{i, j} = g_i(b_{i, j}) \cd s_i^{e_{i, j}}$
for some $b_{i, j} \in \a$ and $e_{i, j} \leq 0$. 
Define the finite sequence
$\bb_i := (b_{i, 1}, \ldots)$ in $\a$. 
We see that the 
sequence $g_i(\bb_i) := \bigl( g_i(b_{i, 1}), \ldots \bigr)$ also generates the 
ideal $\a_{s_i} \sub A_{s_i}$.

Define the finite sequence 
$\bb := \bb_1 \smallsmile \bb_2 \smallsmile \cdots \smallsmile \bb_n$,
the concatenation of the $\bb_i$. Since for every $i = 1, \ldots, n$ the 
sequence $g_i(\bb)$ in $A_{s_i}$ generates the ideal $\a_{s_i}$, it follows 
that the sequence $\bb$ generates the ideal  
$\a \sub A$. It remains to prove that $\bb$ is WPR. 

Fix $p < 0$ and $j \geq 0$. For every $k \geq j$ we consider the  homomorphism
\begin{equation} \label{eqn:60}
\opn{H}^p(\mu_{k, j}) :
\opn{H}^p \bigl( \opn{K}(A; \bsym{b}^k) \bigr) \to 
\opn{H}^p \bigl( \opn{K}(A; \bsym{b}^j) \bigr) . 
\end{equation}
Because the sequence $g_i(\bb)$ in $A_{s_i}$ generates the 
WPR ideal $\a_{s_i}$, Corollary \ref{cor:140} says that this is a WPR 
sequence. Therefore there is some 
$k_i \geq j$ such that the homomorphism
\begin{equation} \label{eqn:325}
\opn{H}^p(\mu_{k_i, j}) : 
\opn{H}^p \bigl( \opn{K}( A_{s_i}; g_i(\bb)^{k_i}) \bigr) \to 
\opn{H}^p \bigl( \opn{K}( A_{s_i}; g_i(\bb)^j) \bigr)
\end{equation}
is zero. Now the homomorphism (\ref{eqn:325}) is gotten from the homomorphism 
(\ref{eqn:60}) by replacing $k$ with $k_i$, and applying the functor 
$A_{s_i} \ot_A (-)$. Hence, taking $k := \opn{max} (k_1, \ldots, k_n)$, 
the homomorphism (\ref{eqn:60}) is zero. So the sequence $\bb$ is WPR. 
\end{proof}

\begin{exa} \label{exa:115}
Suppose the ideal $I \sub A$ defines an {\em effective Cartier divisor} on 
$\opn{Spec}(A)$. This means  
(see \cite[Section
{\href{https://stacks.math.columbia.edu/tag/01WQ}{tag=\texttt{01WQ}}}]{SP}) 
that there is some covering sequence
$(s_1, \ldots, s_n)$ of $A$, such that each of the ideals 
$I_{s_i} := A_{s_i} \ot_A I \sub A_{s_i}$ 
is generated by a single regular element.  
So the ideals $I_{s_i} \sub A_{s_i}$ are all WPR. By Theorem \ref{thm:306} 
it follows that the ideal $I \sub A$ is WPR.
\end{exa}

\section{Weak Proregularity and Prisms} \label{sec:prisms}

For a prime number $p$ let $\Z_p \sub \Q$ denote the local ring of $\Z$ at $p$; 
i.e.\ $\Z_p = \Z_{\p}$ where $\p := (p) \in \opn{Spec}(\Z)$. (This is not to be 
confused with the complete local ring $\what{\Z}_p = \what{\Z}_{\p}$.)

\begin{dfn}[\cite{BS2}] \label{dfn:135}
A {\em $p$-adic prism} is a pair $(A, I)$, where $A$ is a $\Z_p$-ring, and 
$I \sub A$ is an ideal.
The conditions are:
\begin{itemize}
\rmitem{i} The ideal $I$ defines an effective Cartier divisor on 
$\opn{Spec}(A)$. 

\rmitem{ii} The ring $A$ is sequentially derived $\a$-adically complete, where 
$\a := I + (p) \sub A$. 

\rmitem{iii} A condition about a generalized Frobenius lift,  which is not 
relevant to our discussion.
\end{itemize}
\end{dfn}

The ring $A$ is not assumed to be noetherian. However the ideal $I$ is WPR, as 
can be seen in Example \ref{exa:115}. Therefore the ideal $\a$ is finitely 
generated. 

\begin{dfn}[\cite{BS2}] \label{dfn:136}
A $p$-adic prism $(A, I)$ is called {\em bounded} if the ring $A / I$ has 
bounded $p$-torsion.
\end{dfn}

\begin{thm} \label{thm:135}
Suppose $(A, I)$ is a bounded $p$-adic prism. Then the ideal 
$\a := I + (p) \sub A$ is weakly proregular.
\end{thm}

\begin{proof}
Choose a covering sequence $(s_1, \ldots, s_n)$ of $A$, such that for every $k$
the ideal $I_{s_k} := A_{s_k} \ot_A I \sub A_{s_k}$
is generated by a single regular element $b_k \in A_{s_k}$.

Write $\p := (p) \sub A$. 
Let $\mbar{A} := A / I$, 
and let $\mbar{p} \in \mbar{A}$ be the image of $p$.
By assumption the ring $\mbar{A}$ has bounded  $\mbar{p}$-torsion, 
so by Proposition \ref{prop:100} the element $\mbar{p} \in \mbar{A}$ is WPR.

Fix some index $k$. Let $g_k : A \to A_{s_k}$ and
$\bar{g}_k : \mbar{A} \to \mbar{A}_{s_k}$ be the localization ring 
homomorphisms, and define the elements
$p_{k} := g_k(p) \in A_{s_k}$ and 
$\mbar{p}_{k} := \bar{g}_k(\mbar{p}) \in \mbar{A}_{s_k}$.
Note that $\mbar{A}_{s_k} = A_{s_k} / I_{s_k}$, 
and the element $\mbar{p}_{k}$ is the image of the element $p_k$ under the 
surjection $A_{s_k} \to \mbar{A}_{s_k}$. 
Since $\bar{g}_k$ is a flat homomorphism, according to Proposition 
\ref{prop:310} the element 
$\mbar{p}_{k} \in \mbar{A}_{s_k}$ is WPR. 
We see that the pair of elements $(b_k, p_{k})$ in the ring $A_{s_k}$
satisfies the conditions of Theorem \ref{thm:105}. Therefore 
the pair $(b_k, p_{k})$ is a WPR sequence in $A_{s_k}$, and 
the ideal $\a_{s_k} = I_{s_k} + \p_{s_k}\sub A_{s_k}$ that this pair generates 
is WPR.

Finally, by Theorem \ref{thm:306} we conclude that the ideal 
$\a \sub A$ is WPR.
\end{proof}


\end{document}